\documentclass[reqno, 11pt]{amsart}
\usepackage{graphicx, enumerate, amsmath, amssymb,amsthm,manfnt,hyperref}
\usepackage[utf8]{inputenc} 
\newtheorem{thm}{Theorem}[section]
\newtheorem{cor}[thm]{Corollary}
\newtheorem{lem}[thm]{Lemma}
\newtheorem{prop}[thm]{Proposition}
\newtheorem{result}[thm]{Result}

\theoremstyle{definition}

\theoremstyle{remark}

\newcommand{\norm}[1]{\left\Vert#1\right\Vert}
\newcommand{\abs}[1]{\left\vert#1\right\vert}

\newcommand{\rl}{{\mathbb{R}}}
\newcommand{\cx}{{\mathbb{C}}}

\newcommand{\h}{\mathbb{H}}

\newcommand{\dbar}{\overline{\partial}}

\newcommand{\tensor}{\otimes}
\newcommand{\csor}{ \widehat{\otimes} }

\title[Explicitly Computable Cohomology]{Some non-pseudoconvex
domains with explicitly computable non-Hausdorff Dolbeault cohomology}
\author{Debraj Chakrabarti}
\address{Central Michigan University, Mt. Pleasant,  MI 48859, USA}
\email{chakr2d@cmich.edu}
\thanks{ This work was partially supported by a grant from the Simons Foundation (\#316632), and also by an Early Career internal grant from Central Michigan University.}
\begin{document}
\maketitle
\begin{abstract}We explicitly compute the Dolbeault cohomologies of certain domains in complex space generalizing the classical Hartogs figure. The cohomology groups are non-Hausdorff topological vector spaces, and it is possible to identify the reduced (Hausdorff) and the indiscrete part of the cohomology.\end{abstract}
\section{Introduction}
\subsection{Topology of Dolbeault groups} The Dolbeault cohomology groups of a complex manifold are among its 
fundamental invariants. Usually, in complex analysis, motivated  by the two classical nineteenth 
century examples of planar domains and compact Riemann surfaces, we consider manifolds which have ``nice'' cohomology. Important results in this direction (under  extra  ``positivity" or ``convexity"  conditions such as  pseudoconvexity, Stein-ness, $q$-convexity, Kählerness, existence of 
 a positive line bundle, etc.) assert that the cohomology in a certain degree vanishes or is of finite dimension. 
 
For a general complex manifold, there is very little that one can say about the cohomology. Typically, the natural linear 
topology of the Dolbeault group (arising 
as the quotient of the Fréchet topology on the space of smooth $\dbar$-closed forms of a particular degree by the subspace 
of $\dbar$-exact forms) is not Hausdorff, since there is no reason in general why the $\dbar$-operator should have closed range. Manifolds with non-Hausdorff Dolbeault cohomology, of which many examples are known, arise primarily as counter-examples  (e.g. \cite{serre,cassa,Mal75, Demailly78}.) 
Among the simplest 
 examples of such domains is the classical {\em Hartogs figure}, the venerable non-pseudoconvex
 domain contained in the unit bidisc  $\Delta^2\subset\cx^2$ represented as
\begin{equation}\label{eq-hf}
\h_1=\left\{(z_1,z_2)\in \Delta^2\colon \abs{z_1}> \frac{1}{2} \text{ or  } \abs{z_2}<\frac{1}{2}\right\}.
 \end{equation}
 
 There are partial results, especially for domains in $\cx^2$ or in Stein surfaces, which give sufficient condition for the Dolbeault cohomology in certain degrees to be Hausdorff (e.g. \cite{ MR847655, MR3078262}.)  
 Many aspects of the structure of the cohomology groups
 are not well-understood. For example, recall that a  (possibly  non-Hausdorff) topological vector space $E$ can be written as a topological direct sum of 
  subspaces as $E=E_{\rm ind}\oplus E_{\rm red}$, where, in the subspace topology, the linear subspace $E_{\rm ind}$ is {\em indiscrete} 
(the only non-empty open set is the whole space), and the linear subspace $E_{\rm red}$ is Hausdorff (see, e.g.  \cite[Theorem~5.11]{kena}). 
Therefore, the cohomology group has an indiscrete and an Hausdorff direct summand, the latter being referred to as the ``reduced'' cohomology (see \cite{cassa}.)  It is natural to ask what features of the geometry of the domain are reflected 
in this  decomposition of the cohomology into the ``reduced'' and ``indiscrete'' parts. 
 
 As a step towards understanding such phenomena, we compute {\em explicitly} the cohomology of a class of domains in Stein manifolds
 generalizing the Hartogs figure \eqref{eq-hf}.  Let $Z$ be a connected Stein manifold, and let $Z_0$ be a  open 
Stein domain in $Z$ with $Z_0\not=Z$. We call $(Z_0,Z)$ a {\em Stein pair}. For Stein pairs $(X_0,X)$ and $(Y_0,Y)$ we consider the domain $\h\subsetneq X\times Y$, defined as 
 \begin{equation}    \label{eq-ghf}\h= (X\times Y_0) \cup (Y\times X_0),
\end{equation}
which we may refer to as a  {\em generalized Hartogs figure}.
 Denoting the unit disc in the plane by $\Delta$, the domain $\h_1$  in  \eqref{eq-hf} corresponds to $X=Y=\Delta$, $X_0=\left\{z\in X\colon \abs{z}>\frac{1}{2}\right\}$, 
and $Y_0=\left\{z\in Y\colon \abs{z}< \frac{1}{2}\right\}.$

\subsection{Results}We adopt the following notation: if $Z_0$ is a domain in a manifold $Z$, and $\mathcal{Q}$ is a subspace of $\mathcal{O}(Z)$, then  $\mathcal{Q}|_{Z_0}$ denotes the space  of restrictions to 
$Z_0$ of functions in $\mathcal{Q}$.  We think of $\mathcal{Q}|_{Z_0}$ as a subspace of $\mathcal{O}(Z_0)$ with 
the subspace topology.  Recall that the Stein pair $(Z_0,Z)$ is called {\em Runge}, if  $\mathcal{O}(Z)|_{Z_0}$ is dense in $\mathcal{O}(Z_0)$. Also recall
that a topological space is said to be {\em indiscrete} if the only nonempty open set in it is the whole space. A simple class of 
domains with non-Hausdorff cohomology is given by the following result:
\begin{prop}\label{prop-nonhaus} 
If at least one of the pairs $(X_0,X)$ and $(Y_0,Y)$ is Runge, then for each $p$ with $0\leq p \leq \dim \h$,  the space 
$H^{p,1}(\h)$ is either the zero space, or an indiscrete topological vector space of  uncountably infinite dimension. 
Further, there is at least one $p$ in this range for which the latter option holds. 

For $q>1$, we have $H^{p,q}(\h)=0$, and this last fact holds whenever $(X_0,X)$ and $(Y_0,Y)$ are Stein pairs, without
the hypothesis that one of them is Runge.
\end{prop}

In the important special case when $X,Y$ are domains in complex vector spaces  we can say more:
\begin{cor}\label{cor-thm1}
Suppose that the Stein manifolds $X,Y$ are parallelizable. Then for each 
$p$ with $0\leq p \leq \dim \h$ the cohomology $H^{p,1}(\h)$ is an indiscrete topological vector space of uncountable dimension. 
\end{cor}
We now introduce further hypotheses on the pairs $(X_0,X)$ and $(Y_0,Y)$ in \eqref{eq-ghf} so that we can explicitly 
compute the cohomology in certain degrees. We say that the Stein pair $(Z_0,Z)$ is
\begin{enumerate}
\item {\em Split}  if   $\mathcal{O}(Z)|_{Z_0}$ is a closed subspace of $\mathcal{O}(Z_0)$, and
there is a {\em closed} subspace $Q(Z_0,Z)$ of $\mathcal{O}(Z_0)$ such that we have a direct sum representation
\begin{equation}\label{eq-complement}
\mathcal{O}(Z_0)= \mathcal{O}(Z)|_{Z_0} \oplus Q(Z_0,Z).
\end{equation}
\item {\em Quasi-split} if   $\mathcal{O}(Z)|_{Z_0}$ is neither dense nor closed in  $\mathcal{O}(Z_0)$, and there is a closed subspace $Q_r(Z_0,Z)$ of  $\mathcal{O}(Z_0)$ such that 
\begin{equation}\label{eq-complement2}
  \mathcal{O}(Z_0)= \overline{\mathcal{O}(Z)|_{Z_0}} \oplus Q_r(Z_0,Z).
\end{equation}
\end{enumerate}
To illustrate this somewhat uncommon situation, 
let $Z$ be the unit disc $\{\abs{z}<1\}\subset \cx$, and  consider the annuli  $Z_1=\left\{\frac{1}{2}<\abs{z}<1\right\}$ and
$Z_2 = \left\{\frac{1}{2}<\abs{z}<\frac{3}{4}\right\}$. Then it is not difficult to see that the pair $(Z_1,Z)$ is split and the 
the pair $(Z_2,Z)$ is quasi-split. In Propositions~\ref{prop-split} and \ref{prop-quasisplit} below, a wider class of examples
generalizing these are given.

Note also given a split pair $(Z_0,Z)$ (resp. a quasi-split pair $(W_0,W)$) the space $Q(Z_0,Z)$ (resp. $Q_r(W_0,W)$) is not unique, but all such spaces are isomorphic to the quotient TVS $\mathcal{O}(Z_0)/\mathcal{O}(Z)|_{Z_0}$
(resp. to   $\mathcal{O}(W_0)/\overline{\mathcal{O}(W)|_{W_0}}$.)

The main result of this paper is the following:
\begin{thm}\label{thm-split} Assume that the pair $(X_0,X)$ is split. Then we have an  isomorphism of TVS:
\begin{equation}\label{eq-formula}
 H^{0,1}(\h)\cong \frac{Q(X_0,X)\csor \mathcal{O}(Y_0)\phantom{|_{X_0\times Y_0}}}{\left(Q(X_0,X)\csor\mathcal{O}(Y)\right)|_{X_0\times Y_0}}.
\end{equation}
Consequently:
\begin{enumerate}
\item If $(Y_0,Y)$ is Runge,  $H^{0,1}(\h)$ is an indiscrete topological vector space of uncountable dimension. (cf.  Proposition~\ref{prop-nonhaus} above.)
\item If $(Y_0,Y)$ is also split, then $H^{0,1}(\h)$ is Hausdorff and infinite-dimensional, and is isomorphic to ${Q}(X_0,X)\csor{Q}(Y_0,Y)$. 
\item If $(Y_0,Y)$ is quasi-split, then $H^{0,1}(\h)$ is non-Hausdorff, but  not indiscrete. There is a topological direct sum decomposition into  subspaces
\begin{equation}\label{eq-h01quasisplit}
H^{0,1}(\h)= H^{0,1}_{\rm ind}(\h)\oplus H^{0,1}_{\rm red}(\h)
\end{equation}
where $H^{0,1}_{\rm ind}(\h)$ is an indiscrete topological vector space of uncountable dimension, and $H^{0,1}_{\rm red}(\h)$ is a non-zero Hausdorff topological vector space, and are given explicitly as:
\begin{equation}
\label{eq-h31}H^{0,1}_{\rm red}(\h) \cong {Q}(X_0,X)\csor{Q}_{r}(Y_0,Y)
\end{equation}
and 
\begin{equation}\label{eq-h32}
 H^{0,1}_{\rm ind}(\h)\cong
\frac{Q(X_0,X)\csor\overline{\mathcal{O}(Y)|_{Y_0}} 
}{\left(Q(X_0,X)\csor \mathcal{O}(Y)\right)|_{X_0\times Y_0}}.
\end{equation}
\end{enumerate}
\end{thm}

In Section~\ref{sec-examples} below, we use these results to compute the cohomologies of several elementary Reinhardt domains in $\cx^2$. 

\subsection{Acknowledgments}The author gratefully acknowledges the  inspiration and encouragement of
Christine Laurent-Thiébaut and  Mei-Chi Shaw.  Also, many thanks to Franc Forstneric and Bo Berndtsson for interesting 
conversations and comments.

\section{Preliminaries}
\subsection{Natural topology on the cohomology}We collect here a few facts which we will use (see \cite{GuRo}) . Let $\Omega^p$ denote the sheaf of germs of holomorphic $p$-forms on a complex manifold $M$. By  {\em Dolbeault's  theorem} we can naturally identify the Sheaf cohomology group $H^q(M, \Omega^p)$ with the Dolbeault group $H^{p,q}(M)$.  Further, if $\mathfrak{U}$ is an open cover of $M$ by countably many Stein open subsets, by Leray's theorem, there is a natural isomorphism of the \v{C}ech cohomology group $H^q(\mathfrak{U},\Omega^p)$ with the cohomology group $H^q(M, \Omega^p)$. 
Each of the three isomorphic cohomology groups $H^{p,q}(M), H^q(M, \Omega^p)$ and $H^q(\mathfrak{U},\Omega^p)$ is a topological vector space in a
natural way.  Recall that for the Dolbeault group $H^{p,q}(M)$ this topology arises as the quotient topology of the Fréchet topology of the space  $\mathcal{Z}^{p,q}_{\dbar}(M)$ 
of $\dbar$-closed $(p,q)$-forms by the subspace $\dbar(\mathcal{A}^{p,q-1}(M))$ of exact forms. Note also that the space of \v{C}ech cochains
$C^q(\mathfrak{U}, \Omega^p)$ with respect to the cover $\mathfrak{U}$ is the direct sum of Fréchet spaces, and therefore carries a natural topology. 
Denoting by $Z^q(\mathfrak{U}, \Omega^p)$  the subspace of cocycles, the \v{C}ech group $H^q(\mathfrak{U},\Omega^p)$ is the quotient $Z^q(\mathfrak{U}, \Omega^p)/\delta(C^{q-1}(\mathfrak{U},\Omega^p)$, and therefore carries a natural quotient topology. Finally, 
the group $H^q(M, \Omega^p)$ is the direct limit (with respect to refinement of covers) of the \v{C}ech groups $H^q(\mathfrak{U},\Omega^p)$, and therefore
can be endowed with the strong topology, i.e. the finest topology with respect to which the natural map $H^q(\mathfrak{U},\Omega^p) \hookrightarrow H^q(M, \Omega^p)$ is  continuous for each open cover $\mathfrak{U}$ of $M$.

It is a fundamental fact that these natural topologies are essentially the same. More precisely we have the following:
\begin{result}[{\cite[Section 2]{laufer}}] \label{res-lauf}With respect to the
 natural topologies, described above, the Dolbeault and Leray isomorphisms are linear homeomorphisms of topological vector spaces.
\end{result}

This means, in particular, that the decomposition into indiscrete and reduced parts has analogs for the \v{C}ech cohomology, and 
further the Dolbeault and Leray isomorphisms respect this decomposition. For details, see \cite{cassa}.
\subsection{Laufer and Siu Theorems} We now recall an important general fact regarding Dolbeault groups of domains in Stein 
manifolds, which will allow us to conclude that certain topological vector spaces are of uncountable dimension by knowing that  these spaces are non-zero.
\begin{result}[{\cite{laufer_infinite, Siu_uncountable}}]\label{res-laufersiu}
Let $\Omega$ be a non-empty domain in a Stein manifold of positive dimension. Then 
the cohomology group $H^{p,q}(\Omega)$ is either zero or uncountably infinite dimension. If $H^{p,q}(\Omega)=
H^{p,q}_{\rm red}(\Omega)\oplus H^{p,q}_{\rm ind}(\Omega)$ is the decomposition of the cohomology into its 
reduced (Hausdorff) and indiscrete part, then each of these parts is either zero or of uncountably infinite dimension. 
\end{result}

\subsection{Tensor products} For the general theory of nuclear spaces and topological tensor products we refer to \cite{trevesbook}. However, in our application, the tensor products used will be of a very simple type and can be easily 
described in elementary terms. Let $X,Y$ be complex manifolds, and let $\mathcal{Q}\subset \mathcal{O}(X)$ and
$\mathcal{R}\subset\mathcal{O}(Y)$ be linear subspaces. The {\em algebraic tensor product} $\mathcal{Q}\tensor\mathcal{R}$ is the linear span of the functions $f\tensor g$ on $X\times Y$, where by definition
$(f\tensor g)(z,w)=f(z)g(w)$ for $z\in X,w\in Y$.  The topological tensor product $\mathcal{Q}\csor\mathcal{R}$, is the closure of
the algebraic tensor product $\mathcal{Q}\tensor \mathcal{R}$ in the space $\mathcal{O}(X\times Y)$ (this is equivalent 
to the abstract definition in \cite{trevesbook} in this particular case.)  Therefore,
whatever $\mathcal{Q},\mathcal{R}$ might be, $\mathcal{Q}\csor\mathcal{R}$ is a Fréchet space.

\section{Proofs}
\subsection{Proof of Proposition~\ref{prop-nonhaus} and its Corollary~\ref{cor-thm1}} 
\begin{lem}\label{lem-nonstein} The manifold $\h$  in \eqref{eq-ghf} is not Stein. This result holds even if the complex manifolds $X,X_0, Y,Y_0$  are not assumed to be Stein, all we need is that $X_0\subsetneq X$ and $Y_0\subsetneq Y$ are 
proper domains.
\end{lem}
This is a consequence of the fact that if we think of $\h$ as a domain in $X\times Y$, there is a point on the boundary of
$\h$ to which each holomorphic function on $\h$ extends holomorphically. Indeed with more work, we can show that 
each holomorphic function on $\h$ extends to a neighborhood of $\partial X_0\times \partial Y_0$, but this will not be needed for our application.

\begin{proof}[Proof of Lemma~\ref{lem-nonstein}] Let $m=\dim X, n=\dim Y$.
We first consider the special case in which $X=\Delta^m, Y=\Delta^n$, and for some $0<\rho,r<1$, we have $X_0=\Delta_\rho^m, Y_0=\Delta_r^n$, where
$\Delta_\rho^m$ denotes the polydisc $\left\{\abs{z_j}<\rho \text{ for } j=1,\dots,m\right\}$ and $\Delta_r^n$ is similarly defined. Then the logarithmic image
\[\mathcal{L}_\h= \{(\log\abs{z_1},\dots, \log\abs{z_{m+n}})\in \rl^{m+n}\colon (z_1,\dots, z_{m+n})\in \h\}\]
is not convex, and consequently by classical results of function theory, the envelope of holomorphy of $\h$ is the smallest domain in $\Delta^{m+n}$ 
containing $\h$ whose logarithmic image is convex (see e.g. \cite[page  83 ff.] {GrauFri}.) In particular, it follows that each holomorphic function on $\h$ 
extends to a neighborhood of $\partial X_0\times \partial Y_0$.

In the general situation, let $P\in X_0$ be a point near the boundary $\partial X_0$ in the sense that there is a neighborhood $V$ of $P$ in $X$ which is biholomorphic to the polydisc $\Delta^m$ by a biholomorphic map which carries $P$ to 0, and $V$ has nonempty intersection with $\partial X_0$. Identifying
$V$ with $\Delta^m$, let $0<\rho<1$ be the smallest radius such that the intersection $\partial\Delta^m_\rho \cap \partial X_0$ is nonempty. Similarly,
let $W$ be a coordinate neighborhood on $Y$ biholomorphic to $\Delta^n$ centered at a point of $Y_0$, and let $r$ be the smallest radius for which 
the intersection $\partial\Delta^n_r \cap \partial Y_0$ is nonempty. Then define
\[ H=\Delta^m\times \Delta^n_r \cup \Delta^m_\rho \times \Delta^n \subset \h.\]
By the previous paragraph, each holomorphic function on $h$ extends to a neighborhood of $\partial\Delta_\rho^m \times \partial \Delta_r^n$. 
Choosing $P'\in \partial\Delta^m_\rho \cap \partial X_0$ and $Q'\in \partial\Delta^n_r \cap \partial Y_0$, we see that each holomorphic function on $\h$ extends to a neighborhood of $(P',Q')\in \partial \h$, so that $\h$ is not Stein.
\end{proof}

\begin{proof}[Proof of Proposition~\ref{prop-nonhaus}] Let  
\begin{equation}\label{eq-leraycovering}
 U_1 = X_0\times Y, \text{ and   } U_2= X\times Y_0
\end{equation}
Then $\mathfrak{U}=\{U_1,U_2\}$ is an open cover of
of $\h$ by Stein open sets, and is consequently by Leray's theorem, the \v{C}ech group $H^q(\mathfrak{U},\Omega^p)$ is naturally isomorphic to the 
sheaf cohomology group $H^q(\h, \Omega^p)$, where $\Omega^p$ denotes the sheaf of germs of holomorphic $p$-forms on $\h$.  Since the manifold $\h$ 
admits an open cover by two Stein domains, it follows that if $q>1$, then the the space of \v{C}ech cochains
$C^q(\mathfrak{U}, \Omega^p)$ vanishes, and consequently, we have $H^q(\mathfrak{U},\Omega^p)=0$. By the Dolbeault isomorphism theorem, we have 
that $H^{p,q}(\h)\cong H^q(\h, \Omega^p) \cong H^q(\mathfrak{U},\Omega^p)=0$. 

If it was true that $H^{p,1}(\h)=0$ for each $p$, then we will have $H^{p,q}(\h)=0$ for $q>0$. Therefore 
we would have that $\h$ is Stein, which contradicts Lemma~\ref{lem-nonstein}. Therefore, there is at least one $p$, 
for which $H^{p,1}(\h)\not =0$.

Denote by $U_{12}$ the intersection $U_1\cap U_2$, and recall that $\Omega^p$ is the sheaf of germs of holomorphic $p$-forms. By Result~\ref{res-lauf}, 
we have a linear homeomorphism:
\begin{align} H^{p,1}(\h)&\cong \frac{Z^1(\mathfrak{U}, \Omega^p)}{\delta(C^0(\mathfrak{U}, \Omega^p))}\nonumber\\
&=\frac{\Omega^p(U_{12})}{\Omega^p(U_1)|_{U_{12}}+ \Omega^p(U_2)|_{U_{12}}},\label{eq-cohom1}
\end{align}
where $\Omega^p(U_1)|_{U_{12}}$  means the space of  restrictions of  holomorphic $p$-forms on $U_1$ to the subset $U_{12}$ (and similarly for the 
other term in the denominator), and the sum of two subspaces of a vector space is as usual  the linear span of their union. Note that till this point we have not made any use of the hypothesis that one of the Stein pairs $(X_0,X)$ and $(Y_0,Y)$ is Runge.

By hypothesis at least one of the pairs
$(X_0,X)$ and $(Y_0,Y)$ is Runge. Let us say for definiteness that $(Y_0,Y)$ is Runge, i.e., $Y_0$ is a  proper Stein domain in the Stein manifold
$Y$ such that  $\mathcal{O}(Y)|_{Y_0}$ is dense in $\mathcal{O}(Y_0)$. It follows that the pair of Stein manifolds $(X_0\times Y_0, X_0\times Y)= (U_{12},U_1)$ 
is also Runge.  It now follows from a classical result of Stein theory  (see \cite[page 170]{GraRem}) that $\Omega^p(U_1)|_{U_{12}}$ is dense in $\Omega^p(U_{12})$ for each $p$. Consequently, the space of exact \v{C}ech cochains $\delta(C^0(\mathfrak{U}, \Omega^p))= {\Omega^p(U_1)|_{U_{12}}+ \Omega^p(U_2)|_{U_{12}}}$ is dense in the space of closed cochains $Z^1(\mathfrak{U}, \Omega^p)=\Omega^p(U_{12})$. For a particular $p$, we therefore have two possibilities:
\begin{enumerate}
\item ${\Omega^p(U_1)|_{U_{12}}+ \Omega^p(U_2)|_{U_{12}}}=\Omega^p(U_{12})$ so that $H^{p,1}(\h)=0$.
\item ${\Omega^p(U_1)|_{U_{12}}+ \Omega^p(U_2)|_{U_{12}}}\not=\Omega^p(U_{12})$  so that $H^{p,1}(\h)$ is a nonzero indiscrete topological vector space, and we have already shown that there is at least one $p$ for which this option must hold. An appeal to Result~\ref{res-laufersiu} completes the proof.
\end{enumerate}
\end{proof}
\begin{proof}[Proof of Corollary~\ref{cor-thm1}]Let $N=\dim \h$.  From the triviality of the tangent bundle of $\h$, we see that on $\h$, we have an isomorphism of sheaves 
$\displaystyle{ \Omega^p \cong \bigoplus_{j=1}^{ \binom{N}{p}}\Omega^0}$,
which gives rise to an isomorphism at the cohomology level
\begin{equation}\label{eq-hp1}
{ H^{p,1}(\h)\cong \bigoplus_{j=1}^{ \binom{N}{p}}H^{0,1}(\h).}
\end{equation}
There is a $p=p_0$ for which the left hand side is an indiscrete TVS of uncountable dimension. Since $H^{0,1}(\h)$ is linearly homeomorphic to a subspace of $H^{p_0,1}(\h)$, it follows from \eqref{eq-hp1}  that $H^{0,1}(\h)$ is also an indiscrete TVS of uncountable dimension. For for any $p$ with $1\leq p \leq N$, the result follows from \eqref{eq-hp1},
since the direct sum of indiscrete TVS is clearly indiscrete.
\end{proof}

\subsection{Proof of Theorem~\ref{thm-split}} \label{sec-proof2}Recall from \eqref{eq-leraycovering} that the two-element Leray cover $\mathfrak{U}$  that we are using to compute 
cohomologies consists of $U_1=X_0\times Y$, and $U_2=X\times Y_0$. We also have $U_{12}=X_0\times Y_0$.

 By hypothesis, the pair $(X_0,X)$ is split. It is easy to see that this implies that the pairs $(U_{12},U_2)= (X_0\times Y_0, 
 X\times Y_0)$ and $(U_1,X\times Y)=(X_0\times Y, X\times Y)$ are also split, are if $Q(X_0,X)$ is a complement of 
 $\mathcal{O}(X)|_{X_0}$ in $\mathcal{O}(X_0)$, then we have:
 
 \begin{equation}\label{eq-u12} \mathcal{O}(U_{12})= \mathcal{O}(U_2)|_{U_{12}}\oplus Q_0(X_0,X)\csor \mathcal{O}(Y_0).\end{equation}
 Similarly, we have
 \begin{equation}\label{eq-u1} \mathcal{O}(U_{1})= \mathcal{O}(X\times Y)|_{U_{1}}\oplus Q_0(X_0,X)\csor \mathcal{O}(Y).\end{equation}
 Restricting each side of \eqref{eq-u1} to the subset $U_{12}\subset U_1$, we obtain
 \begin{equation}\mathcal{O}(U_{1})|_{U_{12}}= \mathcal{O}(X\times Y)|_{U_{12}}\oplus \left(Q_0(X_0,X)\csor \mathcal{O}(Y)\right)|_{U_{12}}\label{eq-ou1u12}
 \end{equation}
We analyze the relation of the terms of \eqref{eq-ou1u12} with $\mathcal{O}(U_2)|_{U_{12}}$. Observe that we have
 \begin{align}
 \left(Q(X_0,X)\csor \mathcal{O}(Y)\right)|_{U_{12}}\cap \mathcal{O}(U_2)|_{U_{12}}&=
 \left(Q(X_0,X)\csor \mathcal{O}(Y)\right)|_{U_{12}} \cap \mathcal{O}(X\times Y_0)|_{X_0\times Y_0}\nonumber\\
  &\subseteq  Q(X_0,X)\csor \mathcal{O}(Y_0)\cap \mathcal{O}(X)|_{X_0}\csor \mathcal{O}(Y_0)\nonumber\\
  &\subseteq \left(Q(X_0,X)\cap \mathcal{O}(X)|_{X_0}\right)\csor \mathcal{O}(Y_0)
  \nonumber\\
  &=\{0\}\csor \mathcal{O}(Y_0)\nonumber\\
  &=\{0\}.\label{eq-direct1}
 \end{align}
Also, we have
\begin{align} \mathcal{O}(X\times Y)|_{U_{12}}
&\subseteq \mathcal{O}(X)|_{X_0}\csor \mathcal{O}(Y_0)\nonumber\\
&=\mathcal{O}(U_2)|_{U_{12}}.\label{eq-direct2}
\end{align}
Using \eqref{eq-ou1u12}, \eqref{eq-direct1} and \eqref{eq-direct2}, we can compute
\begin{align}
\mathcal{O}(U_1)|_{U_{12}}+ \mathcal{O}(U_2)|_{U_{12}}
&= \left(\mathcal{O}(X\times Y)|_{U_{12}}\oplus \left(Q(X_0,X)\csor \mathcal{O}(Y)\right)|_{U_{12}}\right)+
 \mathcal{O}(U_2)|_{U_{12}}\nonumber\\
 &=\left(Q(X_0,X)\csor \mathcal{O}(Y)\right)|_{U_{12}}+ \mathcal{O}(U_2)|_{U_{12}}\label{eq-absorb}\\
 &=\left(Q(X_0,X)\csor \mathcal{O}(Y)\right)|_{U_{12}}\oplus \mathcal{O}(U_2)|_{U_{12}},\label{eq-disjoint}
 \end{align}
where in \eqref{eq-absorb} we use \eqref{eq-direct2} to absorb the term $\mathcal{O}(X\times Y)|_{U_{12}}$
in $\mathcal{O}(U_2)|_{U_{12}}$, and in \eqref{eq-disjoint}, we use \eqref{eq-direct1} to conclude that the sum is actually a direct sum of closed subspaces. Therefore we have
\begin{align*}
H^{0,1}(\h)
&\cong 
\frac{\mathcal{O}(U)|_{U_{12}}}{\mathcal{O}(U_1)|_{U_{12}}+\mathcal{O}(U_2)|_{U_{12}}}& \text{ (cf. \eqref{eq-cohom1})}\\
&=\frac{ Q(X_0,X)\csor \mathcal{O}(Y_0)\oplus \mathcal{O}(U_2)|_{U_{12}}}{\left(Q(X_0,X)\csor \mathcal{O}(Y)\right)|_{U_{12}}\oplus \mathcal{O}(U_2)|_{U_{12}}} &\text{using \eqref{eq-u12} and \eqref{eq-disjoint} }\\
&\cong \frac{ Q(X_0,X)\csor \mathcal{O}(Y_0)}{\left(Q(X_0,X)\csor \mathcal{O}(Y)\right)|_{U_{12}}},
\end{align*}
Which establishes \eqref{eq-formula}, since $U_{12}=X_0\times Y_0$. We now consider what happens under various further hypotheses on the pair $(Y_0,Y)$:

(1)  First assume that $(Y_0,Y)$ is Runge. We claim that in this case $\left(Q(X_0,X)\csor \mathcal{O}(Y)\right)|_{X_0\times Y_0}$ is dense in $Q(X_0,X)\csor \mathcal{O}(Y_0)$ so that $H^{0,1}(\h)$ is an indiscrete space of infinite dimension. Since the algebraic tensor product $Q(X_0,X)\tensor \mathcal{O}(Y_0)$ is dense in $Q(X_0,X)\csor \mathcal{O}(Y_0)$, it suffices to see that the algebraic tensor product $Q(X_0,X)\tensor \mathcal{O}(Y)|_{Y_0}$ is dense in  $Q(X_0,X)\tensor \mathcal{O}(Y_0)$. But this is obvious since $\mathcal{O}(Y)|_{Y_0}$ is dense by hypothesis in $\mathcal{O}(Y_0)$.

(2) Now assume that $(Y_0,Y)$ is split, and  we have a direct sum decomposition into closed subspaces
\[ \mathcal{O}(Y_0)= \mathcal{O}(Y)|_{Y_0}\oplus Q(Y_0,Y).\]
Also, we have
\[ \left( Q(X_0,X)\csor \mathcal{O}(Y)\right)|_{X_0\times Y_0}= Q(X_0,X)\csor \mathcal{O}(Y)|_{Y_0}.\]
Therefore
\begin{align*}H^{0,1}(\h)&\cong \frac{ Q(X_0,X)\csor \mathcal{O}(Y_0)}{\left(Q(X_0,X)\csor \mathcal{O}(Y)\right)|_{U_{12}}}\\
&= \frac{ Q(X_0,X)\csor \left(\mathcal{O}(Y)|_{Y_0}\oplus Q(Y_0,Y)\right)}{\left(Q(X_0,X)\csor \mathcal{O}(Y)\right)|_{X_0\times Y_0}}\\
&=\frac{ Q(X_0,X)\csor \mathcal{O}(Y)|_{Y_0}\oplus Q(X_0,X)\csor Q(Y_0,Y)}{Q(X_0,X)\csor \mathcal{O}(Y)|_{Y_0}}\\
&= Q(X_0,X)\csor Q(Y_0,Y)
\end{align*}
which completes the proof in this case. Note that  since $Q(X_0,X)\cong \mathcal{Q}(X_0,X)$  and $Q(Y_0,Y)\cong \mathcal{Q}(Y_0,Y)$, it follows we have the invariant representation
\begin{equation}\label{eq-split2invariant}
 H^{0,1}(\h)\cong  \mathcal{Q}(X_0,X)\csor\mathcal{Q}(Y_0,Y).
\end{equation}
(3)  Now suppose that the pair $(Y_0,Y)$ is quasi-split, and we have a direct sum decomposition
\[ \mathcal{O}(Y_0)= \overline{\mathcal{O}(Y)|_{Y_0}} \oplus Q_r(Y_0,Y).\]
Therefore, using \eqref{eq-formula}, we obtain 
\begin{align*} 
H^{0,1}(\h)&\cong  \frac{ Q(X_0,X)\csor \mathcal{O}(Y_0)}{\left(Q(X_0,X)\csor \mathcal{O}(Y)\right)|_{X_0\times Y_0}}\\
&=\frac{ Q(X_0,X)\csor\left( \overline{\mathcal{O}(Y)|_{Y_0}} \oplus Q_r(Y_0,Y)\right)}{\left(Q(X_0,X)\csor \mathcal{O}(Y)\right)|_{X_0\times Y_0}}\\
&=\frac{Q(X_0,X)\csor\overline{\mathcal{O}(Y)|_{Y_0}} 
\oplus Q(X_0,X)\csor Q_r(Y_0,Y)}{\left(Q(X_0,X)\csor \mathcal{O}(Y)\right)|_{X_0\times Y_0}}\\
&= \frac{Q(X_0,X)\csor\overline{\mathcal{O}(Y)|_{Y_0}} 
}{\left(Q(X_0,X)\csor \mathcal{O}(Y)\right)|_{X_0\times Y_0}}\oplus Q(X_0,X)\csor Q_r(Y_0,Y).
\end{align*}
Note that the first term of this direct sum decomposition is the quotient of the Fréchet space $Q(X_0,X)\csor\overline{\mathcal{O}(Y)|_{Y_0}}$ by a dense subspace, and hence is indiscrete, and the second term is the tensor product of two  subspaces of Fréchet spaces, and hence a Hausdorff (indeed a Fréchet) topological vector space. 
The expressions \eqref{eq-h01quasisplit}, \eqref{eq-h31} and \eqref{eq-h32} now follow.

\section{Split and Quasi-split pairs}
In this section we give some simple sufficient conditions guaranteeing that a pair is split or quasi-split.
\begin{prop}\label{prop-split}(a) Let $X\subset\cx^n$ be a bounded star-shaped domain, and   $X_0=U\cap X$, where $U$ is a  neighborhood
in $\cx^n$ of  the \v{S}ilov boundary of $X$. Then the image of the restriction map $\mathcal{O}(X)\to \mathcal{O}(X_0)$ is closed.

(b) Assume further in (a) that $X$ is complete Reinhardt and  $X_0$ is Reinhardt. Then the pair $(X_0,X)$ is split.
\end{prop}
\begin{proof}
It is easy to see that the conclusion of part (a) of the proposition is equivalent to the following: for each compact $K\subset X$, there is a 
$C_K>0$ and a compact $K_0\subset X_0$ such that for each $f\in \mathcal{O}(X)$, we have 
\begin{equation}\label{eq-crestimate}
\norm{f}_K \leq C_K \norm{f}_{K_0}, \end{equation}
where $\norm{\cdot}_K, \norm{\cdot}_{K_0}$ denote the sup norm on $K, K_0$. Let $\Sigma$ be the \v{S}ilov boundary of $X$, and without loss of generality assume that $X$ is star-shaped with respect to the origin.  Let 
$K$ be a compact subset of $X$. Then there is a $\lambda$ with $0<\lambda <1$ such that $K$ is still compactly contained in $\lambda X$. Note that  we have $\lambda \Sigma \cap K=\emptyset$, and setting $K_0=\lambda\Sigma$, we see that $K_0$ is the \v{S}ilov boundary of $\lambda X$. Therefore, we have for each $f\in \mathcal{O}(X)$ that 
$\norm{f}_{K} \leq \norm{f}_{K_0}$, which proves part (a).

For part (b), let $A=\{\alpha\in \mathbb{Z}^n| z^\alpha\in \mathcal{O}(X_0)\}$. Then any $f\in \mathcal{O}(X_0)$
has a Laurent expansion $f(z)= \sum_{\alpha\in A}c_\alpha(f)z^\alpha$, where $c_\alpha(f)\in \cx$. We can write
\begin{align*} f(z)&=\sum_{\alpha\in A}c_\alpha(f)z^\alpha \\&= \sum_{\alpha\in \mathbb{N}^n \cap A}c_\alpha(f)z^\alpha+
\sum_{\alpha\in A \setminus\mathbb{N}^n}c_\alpha(f)z^\alpha\\
&=f_1(z)+f_2(z).\end{align*}
From estimate \eqref{eq-crestimate} it follows that $f_1$ extends to a holomorphic function in $\mathcal{O}(X)$.  If we 
consider the closed subspace of $\mathcal{O}(X_0)$ given by
\[ Q(X_0,X)= \{f\in \mathcal{O}(X_0)| c_\alpha(f)=0 \text{ if } \alpha\in \mathbb{N}^n\},\]
then clearly $f_2\in Q(X_0,X)$, so that we obtain a direct sum decomposition into closed subspaces
\[ \mathcal{O}(X_0)= \mathcal{O}(X)|_{X_0} \oplus Q(X_0,X),\]
which shows that $(X_0,X)$ is split.
\end{proof}
\begin{prop}\label{prop-quasisplit} Suppose that $X_0,X_1$ are Stein domains in a Stein manifold $X$ such that 
$X_0\subset X_1$,  the pair $(X_0,X_1)$ is split and the pair $(X_1,X)$ is Runge. Then the pair $(X_0,X)$ is quasi-split.
\end{prop}
\begin{proof}Since the pair $(X_0,X_1)$ is split, there is a closed subspace $Q(X_0,X_1)$ of $\mathcal{O}(X_0)$ such that
\begin{equation}\label{eq-qs1}
  \mathcal{O}(X_0)=\mathcal{O}(X_1)|_{X_0} \oplus Q(X_0,X_1).
\end{equation}
Define $Q_r(X_0,X)=Q(X_0,X_1)$, and note that since $(X_1,X)$ is Runge, we have $\mathcal{O}(X_1)= \overline{\mathcal{O}(X)|_{X_1}}.$ Therefore:
\begin{align*}\mathcal{O}(X_1)|_{X_0} &=\left(\overline{\mathcal{O}(X)|_{X_1}}\right)|_{X_0}\\
&=\overline{\mathcal{O}(X)|_{X_0}}.
\end{align*}
Therefore, \eqref{eq-qs1} takes the form:
\[  \mathcal{O}(X_0)=\overline{\mathcal{O}(X)|_{X_0}} \oplus Q_r(X_0,X),\]
which shows that $(X_0,X)$ is quasi-split.

\end{proof}
\section{Examples}\label{sec-examples}
We now apply Proposition~\ref{prop-nonhaus} and Theorem~\ref{thm-split} to study the cohomology of some well-known domains of elementary multi-variable complex analysis. Each of these is a  domain  in the unit bidisc $\Delta^2\subset\cx^2$. It suffices to find $H^{0,1}(\h)$, as $H^{1,1}(\h)\cong H^{0,1}(\h)\oplus H^{0,1}(\h)$, and $H^{2,1}(\h)\cong H^{0,1}(\h)$.
We already know that $H^{p,2}(\h)=0$ for each $p$.

Let $\Delta$ be the unit disc in the plane, and for $r>0$, let $ \Delta_r = \{z\in\cx\colon\abs{z}<r\}$
be the disc of radius $r$. Also, for $0\leq r <R\leq \infty$, let
$A(r,R)=\{z\in\cx\colon r<\abs{z}<R\}$ be an annulus of inner radius $r$ and outer radius $R$. We begin by observing the following:
\begin{enumerate}
\item if $0<r<1$, the pair $(\Delta_r,\Delta)$ is Runge. Indeed, the holomorphic polynomials are dense in both the domains.
\item if $0<r<1$, then the pair $(A(r,1),\Delta)$ is split, as one can see from Proposition~\ref{prop-split}. We can take 
\begin{equation}
\label{eq-arinfty}  Q(A(r,1),\Delta)=\left\{f\in \mathcal{O}(A(r,1))\left\vert f(z)= \sum_{\nu=-\infty}^{-1}a_\nu z^\nu\right.\right\}.
\end{equation}
\item If $0<r<R<1$, then the pair $(A(r,R),\Delta)$ is quasi-split. This follows from Proposition~\ref{prop-quasisplit},
since $(X_0,X_1)=(A(r,R),\Delta_R)$ is split and $(X_1,X)=(\Delta_R,\Delta)$ is Runge. In particular, we can take
\begin{equation}\label{eq-qrdef}
Q_r(A(r,R),\Delta)=Q(A(r,R),\Delta_R).
\end{equation}
\end{enumerate} 
We now consider several special cases of the domain $\h$ as defined in \eqref{eq-ghf}. Let $0<r_1,r_2<1$ be fixed in what follows.

(1)  Let 
$ \h_0= \left\{z\in\Delta^2 | \abs{z_1}< r_1, \text{ or } \abs{z_2}< r_2\right\}.$
 $\h_0$ corresponds to the choice $(X_0,X)=(\Delta_{r_1},\Delta)$ and $(Y_0,Y)=(\Delta_{r_2},\Delta)$ in \eqref{eq-ghf}.  By Corollary~\ref{cor-thm1} to Proposition~\ref{prop-nonhaus}, it follows that 
$H^{0,1}(\h_0)$ and $H^{1,1}(\h_0)\cong H^{0,1}(\h_0)\oplus H^{0,1}(\h_0)$ are both indiscrete.

(2)  Let $ \h_1= \left\{z\in\Delta^2 | \abs{z_1}>r_1, \text{ or } \abs{z_2}< r_2\right\},$
generalizing slightly the domain in \eqref{eq-hf}.
This corresponds to taking $(X_0,X)=(A(r,1),\Delta)$ and $(Y_0,Y)=(\Delta_{r_2}, \Delta)$ in \eqref{eq-ghf}. Since $(Y_0,Y)$ is Runge, by Corollary~\ref{cor-thm1} we know that $H^{0,1}(\h_1)$ and 
$H^{1,1}(\h_1)$ are indiscrete. Since $(X_0,X)$ is split, we can obtain more information using Theorem~\ref{thm-split}. We have
\[ H^{0,1}(\h_1)\cong \frac{Q(A(r_1,1),\Delta)\csor \mathcal{O}(\Delta_{r_2})}{({Q}(A(r_1,1),\Delta)\csor \mathcal{O}(\Delta))|_{A(r_1,1)\times \Delta_{r_2}}}.\]
From the explicit description of $Q(A(r_1,1),\Delta)$ in \eqref{eq-arinfty} 
it follows that an element of the numerator
may be represented by a Laurent series 
\[ \sum_{\mu=-\infty}^{-1}\sum_{\nu=0}^\infty a_{\mu,\nu}z_1^\mu z_2^\nu,\]
convergent in the domain $A(r_1,\infty)\times\Delta_{r_2}$, and the denominator is the subspace of series converging in the 
larger domain $ A(r_1,\infty)\times\Delta$. This allows us finally to parametrize the 
cohomology classes in $H^{0,1}(\h_1)$ by the double sequence of  coefficients
$\{a_{\mu,\nu}\}$.

(3)  Let $ \h_2= \left\{z\in\Delta^2 | \abs{z_1}>r_1, \text{ or } \abs{z_2}> r_2\right\},$
which corresponds to $(X_0,X)= (A(r_1,1),\Delta)$ and $(Y_0,Y)=(A(r_2,1),\Delta)$ in \eqref{eq-ghf}.
 Note that $\h_2$ is the difference of two polydiscs $\h_2 = \Delta^2\setminus \overline{(\Delta_{r_1}\times \Delta_{r_2})}$. Since
the compact set $ \overline{\Delta_{r_1}\times \Delta_{r_2}}$ has a basis of Stein neighborhoods, it is known  that $\dbar$ has closed range in $\mathcal{A}^{(0,1)}(\h_2)$, and $H^{0,1}(\h_2)$ is Hausdorff (see \cite[Teorema 3]{MR847655}.) This also follows from our computations:
thanks to Theorem~\ref{thm-split}, part (2), we know that $H^{0,1}(\h_2)$ is  Hausdorff, and   using \eqref{eq-split2invariant}:
\begin{align*} H^{0,1}(\h_2)&\cong \mathcal{Q}(X_0,X)\csor\mathcal{Q}(Y_0,Y)\\
&={Q}(A(r_1,1),\Delta)\csor{Q}(A(r_2,1),\Delta).\end{align*}
This can be identified with the collection of Laurent series 
\begin{equation}\label{eq-laurent}
\sum_{\mu=-\infty}^{-1} \sum_{\nu=-\infty}^{-1}a_{\mu,\nu}z_1^{\mu}z_2^{\nu},
\end{equation}
convergent in the domain $A(r_1,\infty)\times A(r_2,\infty)$, i.e., the cohomology can be identified with a space of holomorphic functions. This may be compared with \cite[page 49, example 4]{GriHar}.

 (4)  Let $R$ be such that $r_2<R<1$, and set  
$\h_3=\{z\in \Delta^2| \abs{z_1}> r_1, \text{ or } r_2< \abs{z_2}<R\}$. 
 This corresponds to the choice $(X_0,X)= (A({r_1},1),\Delta)$ (which is split) and $(Y_0,Y)= (A(r_2,R),\Delta)$ (which is quasi-split.)  We may think of this domain as a Hartogs figure from which a closed polydisc has been excised. 
 Using \cite[Teorema 3]{MR847655}, we see that $H^{0,1}(\h_3)$ is non-Hausdorff, and we will show 
 that this is a domain for which there is a non-zero reduced cohomology, as well as a non-zero indiscrete part of the cohomology, and this is in some sense the combination of the two previous examples. 
 
We explicitly 
 compute $H^{0,1}(\h_3)$   using \eqref{eq-h31}:
\begin{align*}
H^{0,1}_{\rm red}(\h_3)&= Q(X_0,X)\csor Q_r(Y_0,Y)\\
&= Q(A(r_1,1),\Delta)\csor Q_r(A(r_2,R),\Delta)\\
&=Q(A(r_1,1),\Delta)\csor Q(A(r_2,R),\Delta_R),
\end{align*}
where in the last line we have used \eqref{eq-qrdef}. This space can be naturally identified with the space of Laurent series of the form \eqref{eq-laurent} convergent in $A(r_1,\infty)\times A(r_2,\infty)$. Also, using \eqref{eq-h32} we have
\begin{align*}
H^{0,1}_{\rm ind}(\h_3)&\cong
\frac{ Q(X_0,X)\csor\overline{\mathcal{O}(Y)|_{Y_0}} 
}{Q(X_0,X)\csor \mathcal{O}(Y)|_{Y_0}}\\
&= \frac{ Q(A(r_1,1),\Delta)\csor\overline{\mathcal{O}(\Delta)|_{A(r_2,R)}}}{ (Q(A(r_1,1),\Delta)\csor\mathcal{O}(\Delta))|_{A(r_1,1)\times A(r_2,R)}},
\end{align*}
which again can be represented as a quotient of a space of Laurent series by a dense subspace.

\end{document}